\documentclass[11pt,leqno]{article}
\usepackage{amsmath,amssymb,amsfonts}
\newtheorem{theorem}{Theorem}[section]

\newtheorem{definition}[theorem]{Definition}
\newtheorem{remark}[theorem]{Remark}
\newtheorem{notation}[theorem]{Notation}
\newtheorem{example}[theorem]{Example}

\newenvironment{proof}{\mbox{\bf Proof.}}{\mbox{$\dashv$}\bigskip}
\begin{document}
\begin{center}
 {\Large\textbf{Non-predetermined  Model Theory}}\\
 \textbf{(short paper)}\\\vspace{.25in}
{\bf  Rasoul Ramezanian}\\
Department of Mathematical Sciences, \\
Sharif University of Technology,\\
P. O. Box 11365-9415, Tehran, Iran\\
 ramezanian@sharif.edu
\end{center}
\begin{abstract} This article introduce a new model theory call
non-predetermined model theory where functions and relations need
not to be determined already and they are determined through time.
\end{abstract}

\section{Introduction}

A mathematical structure is set of object   with a collection of
distinguished functions, relations, and special elements. For
example, one may consider the structure $N=(\mathbb{N}, +,
\cdot,<, 0,1)$ of natural numbers where $+:\mathbb{N}\times
\mathbb{N}\rightarrow \mathbb{N}$ and $\cdot:\mathbb{N}\times
\mathbb{N}\rightarrow \mathbb{N}$ are functions, $<\subseteq
\mathbb{N}\times \mathbb{N}$ is a relation and $0$ and $1$ are two
distinguished elements in $\mathbb{N}$. Whenever a subject (a
mathematician) wants to know the value of for example $5+7$, he
refers to definition of the function $+$ and according to the
definition he finds out that the value is $12$. The function $+$
is predetermined and is independent of the subject behavior.

In this article, we introduce structures which are
non-predetermined and subject-dependent. In non-predetermined
structures, functions and relations are not needed to be
determined already. In non-predetermined structures, for a given
unary function $f$ and an object $a$ the value of $f(a)$ is not
necessary determined already, and it is determined as soon as the
subject intends to know (or compute) the value $f(a)$.

\section{Non-predetermined Structures}

In this section, we introduce non-predetermined structures.

\begin{notation}~
\begin{itemize}

\item[-] Let $D$ be a set. We define $D^\ast$ to be the set of all
finite sequences (string) over $D$.

\item[-]For each string $\mathrm{s}=\langle
d_1,d_2,...,d_k\rangle$ over $D$, and a set  $A$, we let
$F_\mathrm{s}(A)$ be the set of all functions from
$\{d_1,d_2,...,d_k\}$ to $A$.

\item[-] For two strings $\mathrm{s}=\langle
d_1,d_2,...,d_k\rangle$ and $\mathrm{s}'=\langle
d'_1,d'_2,...,d_n\rangle$ we let $\mathrm{s}.\mathrm{s}'=\langle
d_1,...,d_k,d'_1,...,d'_n\rangle$.

\item[-] We refer to the empty string $\langle\rangle$ by
$\lambda$.

\item[-] For two string $\mathrm{s}$ and $\mathrm{t}$,  we say
$\mathrm{s}\leq \mathrm{t}$ whenever there exists $\mathrm{v}$
such that $\mathrm{t=sv}$.

\end{itemize}
\end{notation}

\begin{definition}
 A \emph{non-predetermined function} from $D$ to $A$ denoted by
$f:D\hookrightarrow A$ is a mapping $H$ where $H$ maps each string
$\mathrm{s}\in D^\ast$ to  a function  $H[\mathrm{s}]\in
F_\mathrm{s}(A)$ with following property:
\begin{itemize}
\item[] for each   $\mathrm{s}=\langle d_1,d_2,...,d_k\rangle$ in
$D^\ast$ and each $d_{k+1}\in D$ we have for all $1\leq i\leq k$,
$H[\mathrm{s}'](d_i)=H[\mathrm{s}](d_i)$ where
$\mathrm{s}'=\langle d_1,d_2,...,d_k,d_{k+1}\rangle$.
\end{itemize}
\end{definition}

 A first order  language  $L$  contains
\begin{itemize}
\item  a  finite set of predicate symbols $\mathcal{R}=\{R_i\mid
i\leq m_1\}$, and a natural number $n_{i}$    as its ary,

\item  a  finite set of function symbols $\mathcal{F}=\{f_i\mid
i\leq m_0\}$, and a natural number $n_i$ as its ary,

\item a  set of constant symbols $\mathcal{C}$.

\end{itemize}

\begin{definition}
A \emph{subject-dependent $L$-structure} $\mathcal{M}$ is given by
following data
\begin{itemize}
\item[(i)] A set $M$ called the universe or the underlying set of
$\mathcal{M}$,

\item[(ii)] A collection
$\mathcal{F}=^{\mathcal{M}}=\{f^{\mathcal{M}}_i\mid i\leq m_0\}$
where each $f^{\mathcal{M}}_i$ is a non-predetermined function
$f^{\mathcal{M}}_i:M^{n_i}\hookrightarrow M$ correspond to symbol
function $f_i\in \mathcal{F}$,

\item[(iii)] A collection of relations
$\mathcal{R}=^{\mathcal{M}}=\{R^{\mathcal{M}}_i\mid i\leq m_1\}$
where each $R^{\mathcal{M}}_i$ is a non-predetermined function
$R^{\mathcal{M}}_i:M^{n_i}\hookrightarrow \{0,1\}$ correspond to
symbol function $R_i\in \mathcal{R}$,

\item[(iv)] A collection of distinguish elements
$\{c^\mathcal{M}_i\mid i\in I_2\}$ correspond to constant symbols
in $\mathcal{C}$.
\end{itemize} A state of the structure $\mathcal{M}$ is
$\mathrm{e}=([\mathrm{s}_0,\mathrm{s}_1, ...,\mathrm{s}_{m_0}],
[\mathrm{s}'_0,\mathrm{s}'_1,...,\mathrm{s}'_{m_1}])$ where
$\mathrm{s}_i,\mathrm{s}'_i$ belong to $M^{n_i}$. Always one and
only one state called the \emph{current state} denoted by
$(\mathcal{M},\mathrm{e_c})$. Initially, we may assume that all
strings in the current state are empty, that is
$\mathrm{e_c}=([\lambda,\lambda,...\lambda],[\lambda,\lambda,...,\lambda])$.

The structure $\mathcal{M}$ is like a black box for the
\emph{subject}.

Suppose that $\mathrm{e_c}=([\mathrm{s}_0,\mathrm{s}_1,
...,\mathrm{s}_{m_0}],
[\mathrm{s}'_0,\mathrm{s}'_1,...,\mathrm{s}'_{m_1}])$.

\begin{itemize}
\item[i.] Whenever the subject wants to know what is the value of
a function $f^{\mathcal{M}}_i$ ($1\leq i\leq m_0$) at a point
$d\in M^{n_i}$, he queries $(f^{\mathcal{M}}_i,d)$ to
$\mathcal{M}$ then the current state of $\mathcal{M}$ changes to
$\mathrm{e_c}=([\mathrm{s}_0,\mathrm{s}_1, ...,
\mathrm{s}_i.\langle d\rangle,...,\mathrm{s}_{m_0}],
[\mathrm{s}'_0,\mathrm{s}'_1,...,\mathrm{s}'_{m_1}])$,  and
$H_i[\mathrm{s}_i.\langle d\rangle](d)$ is return as the value of
$f^{\mathcal{M}}_i(d)$ (where $H_i$ is the corresponding map to
$f^{\mathcal{M}}_i$)..

\item[ii.] Whenever the subject wants to know what is the value of
a relation $R^{\mathcal{M}}_i$ ($1\leq i\leq m_1$) at a point
$d\in M^{n_i}$, he queries $(R^{\mathcal{M}}_i,d)$ to
$\mathcal{M}$ then the current state of $\mathcal{M}$ changes to
$\mathrm{e_c}=([\mathrm{s}_0,\mathrm{s}_1, ...,\mathrm{s}_{m_0}],
[\mathrm{s}'_0,\mathrm{s}'_1,..., \mathrm{s}'_i.\langle
d\rangle,...,\mathrm{s}'_{m_1}])$,  and
$H'_i[\mathrm{s}'_i.\langle d\rangle](d)$ is return as the value
of $R^{\mathcal{M}}_i(d)$  (where $H'_i$ is the corresponding map
to $R^{\mathcal{M}}_i$).

\end{itemize}

We refer to $f^{\mathcal{M}}$ and $R^{\mathcal{M}}$ and
$c^{\mathcal{M}}$ as interpretation of symbols $f$, $R$ and $c$ in
structure $\mathcal{M}$.
\end{definition}

\begin{remark} Note that

\begin{itemize}
\item[1)] the current state of the structure $\mathcal{M}$  is
determined according to behavior of the subject,

\item[2)] All functions $f^{\mathcal{M}}_i$'s and relations
$R^{\mathcal{M}}_i$'s are not \emph{predetermined} in the
structure $\mathcal{M}$.  As subject freely chooses a point $d\in
M^{n_i}$ and a function $f^{\mathcal{M}}_i$ (or a relation
$R^{\mathcal{M}}_i$) the current state of the structure changes
and the value $f^{\mathcal{M}}_i(d)$ ($R^{\mathcal{M}}_i(d)$) is
\emph{determined}.
\end{itemize}

\end{remark}

\begin{example} Consider the structure  $\mathcal{M}=\langle M=\{0,1\}, \mathcal{R}^{\mathcal{M}}=\emptyset,
 \mathcal{R}^{\mathcal{M}}=\{ R_1\},
 \mathcal{C}^{\mathcal{M}}=\emptyset\rangle$ where $R_1$ is a
 non-predetermined unary relation defined below:
 \begin{itemize}
\item[-] $H_{R_1}[0](0)=1$,  $H_{R_1}[\langle 0,1\rangle][1]=0$,

\item[-] $H_{R_1}[1](1)=1$, $H_{R_1}[\langle 1,0\rangle][0]=0$,

\item[-] for each   $\mathrm{s}=\langle d_1,d_2,...,d_k\rangle$ in
$\{0,1\}^\ast$ and each $d_{k+1}\in \{0,1\}$ we have for all
$1\leq i\leq k$,
$H_{R_1}[\mathrm{s}'](d_i)=H_{R_1}[\mathrm{s}](d_i)$ where
$\mathrm{s}'=\langle d_1,d_2,...,d_k,d_{k+1}\rangle$.

 \end{itemize}
The relation $R_1$ initially is not determined. At first, the
subject is in the current state $\mathrm{e}_c=([],[\lambda])$. The
subject choose either 0 or 1 to know that if $R_1$ holds for it.
For example, suppose that the subject choose 0, then $R_1(0)$ is
true and the current state changes to $\mathrm{e}_c=([],[\langle
0\rangle])$. After this time, $R_1(0)$ is determined to be true in
the structure $\mathcal{M}$. Note that if the subject chose 1
instead of 0 (when the current state was
$\mathrm{e}_c=([],[\lambda])$) then whenever after that $R_1(0)$
is determined, it would be false.
\end{example}

For two states $\mathrm{e}=([\mathrm{s}_0,\mathrm{s}_1,
...,\mathrm{s}_{m_0}], [\mathrm{s}'_0,\mathrm{s}'_1,...,
\mathrm{s}'_{m_1}])$ and $\mathrm{e}'=([\mathrm{l}_0,\mathrm{l}_1,
...,\mathrm{l}_{m_0}], [\mathrm{l}'_0,\mathrm{l}'_1,...,
\mathrm{l}'_{m_1}])$, we say $\mathrm{e}\leq \mathrm{e}'$ whenever
for all $i$, $\mathrm{s}_i\leq \mathrm{l}_i$ and
$\mathrm{s}'_i\leq \mathrm{l}'_i$.

\begin{definition}
 TERM is the smallest set containing

\begin{itemize}
\item[] variable symbols,

\item[] constants  symbols in $C$,

\item[]  for each function symbol $f_i\in \mathcal{F}$, if
$t_1,t_2,...,t_{n_i}\in TERM$ then $f(t_1,t_2,...,t_{n_i})$ is a
term.
\end{itemize}

\end{definition}

\begin{definition} Let $\mathrm{e_c}=([\mathrm{s}_0,\mathrm{s}_1, ...,\mathrm{s}_{m_0}],
[\mathrm{s}'_0,\mathrm{s}'_1,..., \mathrm{s}'_{m_1}])$, $f_i\in F$
be a symbol function, and $d\in M^{n_i}$. We say the
interpretation of the symbol $f_i$ for $d$   is \emph{determined}
in state $\mathrm{e_c}$ whenever \begin{itemize}\item[-] $d$ is an
element of the finite sequence $\mathrm{s}_i$. \end{itemize} The
interpretation is defined to be
$f^{(\mathcal{M},\mathrm{e_c})}_i(d)=H_i[\mathrm{s}_i](d)$.

\noindent We say the interpretation of the symbol $R_i$ for $d$ is
\emph{determined} in state $\mathrm{e_c}$ whenever
\begin{itemize}\item[-] $d$ is an element of the finite sequence
$\mathrm{s}'_i$. \end{itemize} The interpretation is defined to be
$R^{(\mathcal{M},\mathrm{e_c})}_i(d)=H_i[\mathrm{s}'_i](d)$.

 \end{definition}

Let $t$ be a term using variable $\bar{v}=(v_1,...,v_m)$, and
$\bar{a}=(a_1,...,a_m)$ where $a_i\in M$.  Let
$\mathrm{e_c}=([\mathrm{s}_0,\mathrm{s}_1, ...,\mathrm{s}_{m_0}],
[\mathrm{s}'_0,\mathrm{s}'_1,..., \mathrm{s}'_{m_1}])$. For a
subterm $t'$ of $t$, we inductively define interpretation
$t'^{\mathrm{e_c}}(\bar{a})$ as follows:
\begin{itemize}
\item If $t'$ is a constant symbol $c$, then
$t'^{\mathrm{e_c}}(\bar{a})$ is determined in state $\mathrm{e_c}$
and $t'^{\mathrm{e_c}}(\bar{a}):= c^\mathcal{M}$.

\item  If $t'$ is the variable $v_i$, then
$t'^{\mathrm{e_c}}(\bar{a})$ is determined in state $\mathrm{e_c}$
and $t'^{\mathrm{e_c}}(\bar{a}):=a_i$.

\item If $t'$ is the  term $f_i(t_1,...,t_{n_i})$, then
$t'^{\mathrm{e_c}}(\bar{a})$ is determined in state $\mathrm{e_c}$
whenever
\begin{itemize}
\item[1.] for $1\leq j\leq n_i$, $t_j^{\mathrm{e_c}}(\bar{a})$ is
determined in state $\mathrm{e_c}$,

\item[2.]
$(t_1^{\mathrm{e_c}}(\bar{a}),...,t_{n_i}^{\mathrm{e_c}}(\bar{a}))$
is an element of the finite sequence $\mathrm{s}_i$.
\end{itemize} If $t'^{\mathrm{e_c}}(\bar{a})$ is determined then
$t'^{\mathrm{e_c}}(\bar{a}):=f^{(\mathcal{M}_i,\mathrm{e_c})}(t_1^{\mathrm{e_c}}(\bar{a}),...,t_{n_i}^{\mathrm{e_c}}(\bar{a}))$.

\end{itemize}

% The interpretation of a term $t$, denoted by $t^\mathcal{M}$ is
 %defined to be  a  non-predetermined function from $M^k$ to $M$ for some
 %$k$, similar to the interpretation of terms in model theory (see
 %definition~1.1.4~of~\cite{kn:model}). The only difference is that
 %the interpretations are non-predetermined  functions.

\begin{definition}
FORMULA is the smallest set satisfying the following conditions:
\begin{itemize}
%\item[1)] $\perp\in Formula$,

\item[1)] $t_1,t_2\in TERM$ then $t_1=t_2\in FORMULA$,

\item[2)] for each predicate  symbol $R_i\in \mathcal{R}$, if
$t_1,t_2,...,t_{n_i}\in TERM$ then $R(t_1,t_2,...,t_{n_i})\in
FORMULA$,

\item[3)] $\varphi,\psi\in FORMULA$ then $\neg\varphi,
\varphi\wedge\psi, \varphi\vee\psi, \varphi\rightarrow\psi,
\forall y\varphi, \exists y\varphi \in FORMULA$. We call formulas
defined via item 1,2, and 3 \emph{atomic formulas}.
\end{itemize}
\end{definition}
Let $\varphi(v_1,v_2,...,v_n)$ be a formula with free variables
$v_1,v_2,...,v_n$. Similar to definition~1.1.6~\cite{kn:dm}, we
define what it means for $\varphi(v_1,v_2,...,v_n)$ to hold of
$(a_1,a_2,...,a_n)\in M^n$.

\begin{definition}
Let $\varphi$ be a formula with free variables
$\bar{v}=(v_1,v_2,...,v_n)$, and let $\bar{a}=(a_1,a_2,...,a_n)\in
M^n$. Let $\mathrm{e_c}=([\mathrm{s}_0,\mathrm{s}_1,
...,\mathrm{s}_{m_0}], [\mathrm{s}'_0,\mathrm{s}'_1,...,
\mathrm{s}'_{m_1}])$. We inductively define
$(M,\mathrm{e_c})\models \varphi(\bar{a})$ as follows:

\begin{itemize}
\item[i.] If $\varphi$ is $t_1=t_2$ then $(M,\mathrm{e_c})\models
\varphi(\bar{a})$ whenever

 for all states $\mathrm{e}$ after
$\mathrm{e_c}$ ($\mathrm{e_c}\leq \mathrm{e}$) if  both
$t_1^{\mathrm{e }}(\bar{a})$ and $t_2^{\mathrm{e}}(\bar{a})$ are
determined in state $\mathrm{e}$ then
$t_1^{\mathrm{e}}(\bar{a})=t_2^{\mathrm{e}}(\bar{a})$.
%In other
%words:

%$(M,\mathrm{e_c})\models \varphi(\bar{a}):= (FORALL~ \mathrm{e}
%\geq \mathrm{e}_c) [(det(t_1^{\mathrm{e }}(\bar{a})) ~AND~
%det(t_2^{\mathrm{e }}(\bar{a})))\Rightarrow
%t_1^{\mathrm{e}}(\bar{a})=t_2^{\mathrm{e}}(\bar{a})]$.

%\noindent Where $det(t^{\mathrm{e}})$ means that the term
%$t^{\mathrm{e}}$ is determined in state $\mathrm{e}$.

\item[ii.] If $\varphi$ is $R_i(t_1,t_2,...,t_{n_i})$ then
 $(M,\mathrm{e_c})\models \varphi(\bar{a})$  whenever

 for all states $\mathrm{e}$ after
$\mathrm{e_c}$ ($\mathrm{e_c}\leq \mathrm{e}$) if

\begin{itemize}
\item  $t_1^{\mathrm{e}}(\bar{a})$, $t_2^{\mathrm{e}}(\bar{a})$,
..., and $t_{n_i}^{\mathrm{e}}(\bar{a})$ are determined in state
$\mathrm{e}$, and \item
$R^{(\mathcal{M},\mathrm{e})}_i(t_1^{\mathrm{e}}(\bar{a}),...,t_{n_i}^{\mathrm{e}}(\bar{a}))$
is determined in state $\mathrm{e}$,
\end{itemize} then
$R_i^{(\mathcal{M},\mathrm{e})}(t_1^{\mathrm{e}}(\bar{a}),...,t_{n_i}^{\mathrm{e}}(\bar{a}))=1$

\item[iii.] If $\varphi$ is $\neg \psi$  then
$(M,\mathrm{e_c})\models \varphi(\bar{a})$ whenever

 for all states
$\mathrm{e}$ after $\mathrm{e_c}$ ($\mathrm{e_c}\leq \mathrm{e}$)
then $(M,\mathrm{e})\not\models \psi(\bar{a})$.

\item[iv.] If $\varphi$ is $\psi\wedge \theta$ then
$(M,\mathrm{e_c})\models \varphi(\bar{a})$ whenever

\begin{itemize}\item[-] for all states $\mathrm{e}$ after $\mathrm{e_c}$
($\mathrm{e_c}\leq \mathrm{e}$), if atomic formulas which are
subformula of $\psi(\bar{a})$ are determined in state $\mathrm{e}$
then $(M,\mathrm{e})\models \psi(\bar{a})$ \emph{\underline{and}
}\item[-] for all states $\mathrm{e}$ after $\mathrm{e_c}$
($\mathrm{e_c}\leq \mathrm{e}$) if atomic formulas which are
subformula of $\theta(\bar{a})$ are determined in state
$\mathrm{e}$ then $(M,\mathrm{e})\models \theta(\bar{a})$.
\end{itemize}

\item[v.] If $\varphi$ is $\psi \vee\theta$   then
 $(M,\mathrm{e_c})\models \varphi(\bar{a})$ whenever
 \begin{itemize}
\item[-] either for all states $\mathrm{e}$ after $\mathrm{e_c}$
($\mathrm{e_c}\leq \mathrm{e}$), if atomic formulas which are
subformula of $\psi(\bar{a})$ are determined in state $\mathrm{e}$
then $(M,\mathrm{e})\models \psi(\bar{a})$, \emph{\underline{or}}

\item[-]  for all states $\mathrm{e}$ after $\mathrm{e_c}$
($\mathrm{e_c}\leq \mathrm{e}$), if atomic formulas which are
subformula of $\theta(\bar{a})$ are determined in state
$\mathrm{e}$ then $(M,\mathrm{e})\models \theta(\bar{a})$,
 \end{itemize}

\item[vi.] If $\varphi$ is $\exists x \psi(\bar{v},x)$ then
$(M,\mathrm{e_c})\models \varphi(\bar{a})$ whenever

 there exists
$b\in M$ such that for all states $\mathrm{e}$ after
$\mathrm{e_c}$ ($\mathrm{e_c}\leq \mathrm{e}$), if atomic formulas
which are subformula of $\psi(\bar{a},b)$ are determined in state
$\mathrm{e}$ then $(M,\mathrm{e})\models \psi(\bar{a},b)$

\item[vii.]  If $\varphi$ is $\forall x \psi(\bar{v},x)$ then
$(M,\mathrm{e_c})\models \varphi(\bar{a})$ whenever

for all $b\in M$
  for all states $\mathrm{e}$ after $\mathrm{e_c}$
($\mathrm{e_c}\leq \mathrm{e}$), if atomic formulas which are
subformula of $\psi(\bar{a},b)$ are determined in state
$\mathrm{e}$ then $(M,\mathrm{e})\models \psi(\bar{a},b)$.

\item[viii.] If $\varphi$ is $\psi\rightarrow\theta$ then
$(M,\mathrm{e_c})\models \varphi(\bar{a})$ whenever

  for all states $\mathrm{e}$ after $\mathrm{e_c}$
($\mathrm{e_c}\leq \mathrm{e}$), if atomic formulas which are
subformula of $\psi(\bar{a})$ are determined in state
$\mathrm{e}$, and atomic formulas which are subformula of
$\theta(\bar{a})$ are determined in state $\mathrm{e}$, then if
$(M,\mathrm{e})\models \psi(\bar{a})$ then $(M,\mathrm{e})\models
\theta(\bar{a})$.

\end{itemize}
\end{definition}

The main difference between Kripke structures and
non-predetermined structures is that the actual state in Kripke
structure is fixed, but the actual state of non-predetermined
structures (called current state) changes thorough time up to
subject.

\begin{theorem} For all $\phi\in Formula$, and a
subject-dependent structure $\mathcal{M}$, for every states
$\mathrm{e}$ and $\mathrm{e}'$  we have:
\begin{center}
$\mathrm{e}\leq \mathrm{e}'$ and $(\mathcal{M},\mathrm{e})\models
\phi$ then $(\mathcal{M},\mathrm{e}')\models \phi$.
\end{center}

\end{theorem}\begin{proof} It is straightforward.
\end{proof}

We prove that our proposed model theory is sound with respect of
intuitionistic deduction   system~(see page~40 of~\cite{kn:TD}).
For a set of formula $\Gamma$, and a formula $\phi$, we write
$\Gamma\vdash_i \phi$ to say that $\phi$ is derivable form
$\Gamma$ using intuitionistic deduction   system.

\begin{theorem}
Let $\Gamma\vdash_i \phi$ and $\mathcal{M}$ be a structure. Assume
for all formula $\psi\in \Gamma$, $(\mathcal{M},\mathrm{e})\models
\psi$ (abbreviated by $(\mathcal{M},\mathrm{e})\models \Gamma$).
Then $(\mathcal{M},\mathrm{e})\models \phi$.
\end{theorem}
\begin{proof} The proof is similar to the proof of
theorem~5.10, page~82 of~\cite{kn:TD}. It is done by induction on
the derivations in ntuitionistic deduction   system. Let a
derivation terminates with a rule

\begin{center}
$\frac{\Gamma_1\vdash_i \psi_1~~~~~\Gamma_2\vdash_i
\psi_2~~~~~....~~~~~\Gamma_n\vdash_i \psi_n}{\Gamma\vdash_i
\phi}$.
\end{center}
By induction we assume that if $(\mathcal{M},\mathrm{e})\models
\Gamma_i$ then $(\mathcal{M},\mathrm{e})\models \psi_i$, then
using the assumption we prove that if
$(\mathcal{M},\mathrm{e})\models \Gamma$ then
$(\mathcal{M},\mathrm{e})\models \phi$.

\end{proof}

\section{Conclusion}
We introduced structures which functions and relations are not
necessary predetermined and the value of them is eventually
recognized by the way that the subject interacts with the
structure. One may ask that what is the use of non-predetermined
structures.  Suppose that $\Gamma$ is a set of formula, $\phi$ is
a formula, and we want to show that $\Gamma\not \vdash_i \phi$.
One way to show this is to find a model $\mathcal{M}$, and prove
that $\mathcal{M}\vDash \Gamma$ and $\mathcal{M}\not\vDash\phi$.
If constructing a predetermined model is difficult then we may try
to construct non-predetermined model.

\end{document}